\documentclass[a4paper,twoside,12pt]{article}
\usepackage{amssymb,amsmath,amsthm,amsfonts,latexsym}
\usepackage{graphicx,caption,subcaption}
\usepackage[pdftex,bookmarks,colorlinks=false]{hyperref}
\usepackage[hmargin=1.2in,vmargin=1.2in]{geometry}
\usepackage[mathscr]{euscript}
\usepackage{float}

\usepackage{ltablex,booktabs,ragged2e}
\usepackage{array,multirow,longtable}
\usepackage[tableposition=below]{caption}

\usepackage[auth-lg,affil-it]{authblk}
\numberwithin{equation}{section}

\allowdisplaybreaks

\def\ni {\noindent}

\newcommand{\N}{\mathbb{N}}

\newcommand{\cC}{\mathcal{C}}

\newtheorem{theorem}{Theorem}[section]
\newtheorem{definition}{Definition}[section]
\newtheorem{lemma}[theorem]{Lemma}
\newtheorem{proposition}[theorem]{Proposition}
\newtheorem{corollary}[theorem]{Corollary}

\title{\textbf{Rainbow Neighbourhoods of Graphs}}

\author{Johan Kok}
\affil{\small Tshwane Metropolitan Police Department, \\ City of Tshwane, South Africa\\ {\tt kokkiek2@tshwane.gov.za}}

\author{Sudev Naduvath}
\affil{\small Centre for Studies in Discrete Mathematics\\ Vidya Academy of Science \& Technology\\ Thrissur, India.\\ {\tt sudevnk@gmail.com}}

\author{Muhammad Kamran Jamil}
\affil{\small Department of Mathematics\\ Riphah Institute of Computing and Applied Sciences\\ Riphah International University\\ Lahore, Pakistan.\\ {\tt m.kamran.sms@gmail.com}}

\date{}

\begin{document}
\maketitle

\begin{abstract}
\ni In this paper, we introduce the notion of the rainbow neighbourhood and a related graph parameter namely the rainbow neighbourhood number and report on preliminary results thereof. The closed neighbourhood $N[v]$ of a vertex $v \in V(G)$ which contains at least one coloured vertex of each colour in the chromatic colouring of a graph is called a rainbow neighbourhood. The number of rainbow neighbourhoods in a graph $G$ is called the rainbow neighbourhood number of $G$, denoted by $r_\chi(G)$. We also introduce the concepts of an expanded line graph of a graph $G$ and a $v$-clique of $v \in V(G)$.  With the help of these new concepts, we also establish a necessary and sufficient condition for the existence of a rainbow neighbourhood in the line graph of a graph $G$.
\end{abstract}

\ni\textbf{Keywords:} colour cluster; colour classes; rainbow neighbourhood; expanded line graph; $v$-clique.

\vspace{0.25cm}

\ni \textbf{Mathematics Subject Classification:} 05C07, 05C38, 05C75, 05C85.
 
\section{Introduction}

For general notation and concepts in graphs and digraphs see \cite{BM1,CL1,FH,DBW}. For Further definitions in chromatic graph theory, see \cite{CZ1,JT1}. Unless specified otherwise, all graphs mentioned in this paper are simple, connected  and undirected graphs. 

\vspace{0.2cm}

The \textit{vertex colouring} or simply a \textit{colouring} of a graph is an assignment of colours or labels to the vertices of a graph subject to certain conditions. In a proper colouring of a graph, its vertices are coloured in such a way that no two adjacent vertices in that graph have the same colour. The chromatic number $\chi(G)\geq 1$ of a graph $G$ is the minimum number of distinct colours that allow a proper colouring of $G$. Such a colouring is called a \textit{chromatic colouring}.

\vspace{0.2cm}

Different types of graph colourings and related parameters have been introduced in various studies on graph colourings. Many practical and real life situations paved paths to different graph colouring problems. In this paper, we study the graphs admitting chromatic colourings subject to certain conditions.

\section{Rainbow Neighbourhoods in a Graph}

Unless mentioned otherwise, we follow the convention that we consider the colouring $(c_1,c_2,c_3,\ldots,c_\ell )$, $\ell = \chi(G)$ in such a way that the colour $c_1$ is assigned to maximum number of vertices, then the colour $c_2$ is given to maximum number of remaining vertices among the remaining uncoloured vertices and proceeding like this, at the final step, the remaining uncoloured vertices are given colour $c_\ell$. This convention may be called the \textit{rainbow neighbourhood convention}.

\vspace{0.2cm}

In view of the convention mentioned above, the notion of a rainbow neighbourhood in a graph is defined as follows.

\begin{definition}{\rm 
Let $G$ be a graph with a chromatic colouring $\cC$ defined on it. The \textit{rainbow neighbourhood} in $G$ is the closed neighbourhood $N[v]$ of a vertex $v \in V(G)$ which contains at least one coloured vertex of each colour in the chromatic colouring $\cC$ of $G$.}
\end{definition}

If $N[v]$ is a rainbow neighbourhood, then we say that vertex $v$ yields a rainbow neighbourhood. It is interesting to observe that the number of vertices which yield rainbow neighbourhoods differ significantly even within the same graph classes. This fact makes the study on the number of vertices yielding rainbow neighbourhoods in different graph classes much interesting. Next we define the rainbow neighbourhood number of a graph.

\begin{definition}{\rm 
Let $G$ be a graph with a chromatic colouring $\cC$ defined on it. The number of vertices in $G$ yielding rainbow neighbourhoods is called the \textit{rainbow neighbourhood number} of the graph $G$, denoted by $r_\chi(G)$.}
\end{definition}

It is very interesting to note that after in-depth study a sensible colouring cohesion index can be discovered. Intuitively such colouring cohesion index could have a similar application for colouring as connectivity has for graphs. This motivates us to proceed further in this direction 
by studying the above mentioned notions.

The following results discuss the rainbow neighbourhood number of some fundamental graph classes. 

\begin{proposition}\label{Prop-2.1a}
For $n\ge 1,\ r_\chi(P_n)=n$.
\end{proposition}
\begin{proof}
Since $\chi(P_1)=1$, the result is trivial. For $n\geq 2$, we have $\chi(P_n)=2$, since any vertex $v \in V(P_n)$ with colour $c_1$ is adjacent to at least one vertex coloured $c_2$ and vice versa. Hence, the result follows.
\end{proof}

\begin{proposition}\label{Prop-2.1b}
For $n\ge 3,\ r_\chi(C_n)=
\begin{cases}
3, & \text {if $n$ is odd},\\
n, & \text {if $n$ is even.}
\end{cases}$
\end{proposition}
\begin{proof}
\textit{Case-1:} Let $n$ be even. Then $\chi(C_n)=2$. Then, the result for cycle $C_n$ follows similarly to that of the path $P_n$. 

\textit{Case-2:} Let $n$ be odd. Then, $\chi(C_n)=3$ and only one vertex in $C_n$ has the colour $c_3$. All other vertices are alternatively coloured using the colours $c_1$ and $c_2$. Let the vertices be labeled clockwise and consecutively $v_1,v_2,v_3,\ldots,v_n$. Without loss of generality, assume that $v_j$, $j \in \{2,3,4,\dots,\ell-1\}$ is coloured $c_3$. Clearly, only the vertices $v_{j-1},v_{j},v_{j+1}$ has closed neighbourhoods containing all colours $c_1,c_2,c_3$. Hence the result follows.
\end{proof}

\begin{proposition}\label{Prop-2.1c}
For $n\ge 1,\ r_\chi(K_n)=n$.
\end{proposition}
\begin{proof}
For $K_n$, we have $N[v] = V(K_n),\ \forall\, v \in V(K_n)$. Hence, the result follows.
\end{proof}

\begin{proposition}\label{Prop-2.1d}
For $n\ge 3$, the rainbow neighbourhood number of a wheel graph $W_{n+1}=C_n+K_1$ is given by
 $r_\chi(W_{n+1}) =
 \begin{cases}
 4, & \text {if $n$ is odd},\\
 n+1, & \text {if $n$ is even.}
 \end{cases}$
\end{proposition}
\begin{proof}
Since, the central vertex, say $v$ of $W_{n+1}$ is adjacent to all vertices of $C_n$ and hence $N[v]$ is clearly a rainbow neighbourhood in $W_{n+1}$. Therefore, $\chi(W_{n+1})=\chi(C_n)+1$ and the proof the result follows immediately.
\end{proof}

\begin{proposition}\label{Prop-2.1e}
For a complete $\ell$-partite graphs $K_{r_1,r_2,\dots,r_\ell},\ \ell \geq 2$, $r_i \geq 1$, $r_\chi(K_{r_1,r_2,\dots,r_\ell})=\sum\limits_{i=1}^{\ell}r_i$.
\end{proposition}
\begin{proof}
For $\ell \geq 2,\ r_i \geq 1$, the complete $\ell$-partite graphs $K_{r_1,r_2,\dots,r_\ell}$  has vertex partitioning $\mathcal{P}_i,\  1\leq i\leq \ell$ and vertices $v_{i,j} \in \mathcal{P}_i$, $1\leq j \leq r_i$ may have the same colour $c_i$. Since vertex $v_{i,j}$ is adjacent to all vertices $v_{m,k} \in \mathcal{P}_m$, $1\leq k\leq r_m$ for $i \neq m$ it follows that $N[v_{i,j}]$ yields a rainbow neighbourhood. Hence, the result, $r_\chi(K_{r_1,r_2,\dots,r_\ell}) = \sum\limits_{i=1}^{\ell}r_i$.
\end{proof}

\begin{proposition}\label{Prop-2.1f}
For a ladder graph $L_n=P_n\Box P_2,\ n \geq 3$ , we have $r_\chi(L_n)=2n$.
\end{proposition}
\begin{proof}
Let the poles of the ladder be $P^{(1)}_n$ and $P^{(2)}_n$ with vertices consecutively labelled top to bottom, $v_{1,1},v_{1,2},v_{1,3},\dots,v_{1,n}$ and $v_{2,1},v_{2,2},v_{2,3},\dots,v_{2,n}$ respectively. Let the steps be the edges $v_{1,i}v_{2,i}$, $2\leq i \leq n-1$. Colour the vertices of $P^{(1)}_n$ using the rule $c(v_{1,i})=c_1$ for odd values of $i$ and $c(v_{1,i})=c_2$ for even values of $i$. Also, let $c(v_{2,i})=c_1$, if $i$ is even and $c(v_{1,i})=c_2$, if $i$ is odd. It can now easily follow that any vertex $v\in V(L_n)$ with colour, say $c_1$, is adjacent to at least one vertex coloured $c_2$ and vice versa. This colouring is obviously a chromatic colouring of $L_n$. Hence, the result follows immediately.
\end{proof}

\section{Some New Results on $r_\chi(G)$}

In this section we present preliminary results in respect of the rainbow neighbourhood number for certain graphs. We begin the study by stating some important results before we proceed to specific graphs. We recall that a graph $G$ of order $n$ has $\chi(G) = n$ if and only if $G$ is complete.

\begin{theorem}\label{Thm-2.1}
Any graph $G$ of order $n$ has $\chi(G) \leq r_\chi(G) \leq n$.
\end{theorem}
\begin{proof}
For $n=1$ the graph $K_1$ has $\chi(K_1)=1$ and $r_\chi(K_1) =1$.  For $n=2$, the path $P_2$ has $\chi(P_2) =2$ and $r_\chi(P_2)=2$. For $n=3$, the path $P_3$ has $\chi(P_3) = 2$ and $r_\chi(P_3) = 3$. Also the cycle $C_3$ has $\chi(C_3)=3$ and $r_\chi(C_3) =3$. Therefore, $r_\chi(G) \geq \chi(G)$ for all graphs of order $1\leq n \leq 3$.

\vspace{0.2cm}

Assume the result holds for all graphs of order $1\leq n\leq k$. Consider graph of order $k$. Hence $r_\chi(G) \geq \chi(G)$. Now attach a new vertex $u$ to a number say, $t$, $1\leq t \leq k$ vertices of $G$ to obtain a new graph $G'$. If possible, identify a colour in the chromatic colouring of $G$ with which vertex $u$ can be coloured such that $G'$ with or without recolouring of vertices, has a chromatic colouring. It implies that $\chi(G') = \chi(G)$ and also $r_\chi(G')\geq r_\chi(G)$. Hence the result holds for the graph $G'$. Alternatively, an additional colour is indeed required to allow a chromatic colouring for $G'$ hence, $\chi(G') = \chi(G) + 1$. Then $N[u]$ yields a rainbow neighbourhood containing $\chi(G)+1$ colours. Further to that, then $N[v]$ yields a rainbow neighbourhood $\forall$ $v\in N(u)$ because the induced subgraph $\langle N[u]\rangle$ is necessary complete (a clique). Therefore, $r_\chi(G') = \chi(G) + 1 = \chi(G')$. So the result $r_\chi(G') \geq \chi(G')$ holds. Through induction the result then holds for all graphs of order $n \in \Bbb N$.\\\\  
Also, $r_\chi(G) \leq n$ is obvious, therefore the result holds. 
\end{proof}

Note that Theorem \ref{Thm-2.1} holds for any graph and not only for the main class of graphs under study.
 
\begin{theorem}\label{Thm-2.2}
If $G$ is a bipartite graph on $n$ vertices, $r_\chi(G)=n$. 
\end{theorem}
\begin{proof}
For any bipartite graph $G$, we have $\chi(G)=2$ and hence the result is an immediate consequence of the proof of Proposition \ref{Prop-2.1e} and Theorem \ref{Thm-2.1}. 
\end{proof}

We observe that if it is possible to permit a chromatic colouring of any graph $G$ of order $n$ such that the star subgraph of $G$, whose vertex set is $N[v]$ such that $v$ is the central vertex and the open neighbourhood $N(v)$ is the set of pendant vertices, has at least one coloured vertex from each colour for all $v \in V(G)$, then $r_\chi(G)=n$. %Certainly to examine this property for any given graph is complex.

\begin{lemma}\label{Lem-2.3}
For any graph $G$, the graph $G'= K_1+G$ has $r_\chi(G')=1+ r_\chi(G)$.
\end{lemma}
\begin{proof}
Let the graph $G$ permit the chromatic colouring $(c_1,c_2,c_3,\dots,c_{\ell})$, where $\ell=\chi(G)$. Certainly, we note that the graph $G'=K_1+G$ requires a chromatic colouring $(c_1,c_2,c_3,\dots,c_{\ell},c_{\ell+1})$. Without loss of generality, let $c(K_1) = c_{\chi(G)+1}$. Obviously, any vertex $v\in V(G)$ that yields a rainbow neighbourhood in $G$ will include the colour $c_{\ell+1}$ in its new closed neighbourhood in $K_1+G$. Hence, a vertex $v \in V(G)$ which yields rainbow neighbourhood in $G$, yields a rainbow neighbourhood in $K_1+G$ also. 

On the other hand, a vertex $u \in V(G)$ that does not yield a rainbow neighbourhood in $G$ will include the colour $c_{\ell+1}$ in its new closed neighbourhood in $K_1+G$, but not all colours in $(c_1,c_2,c_3,\dots,c_{\chi(G)},c_{\ell+1})$. Hence, those vertices $u \in V(G)$, which do not yield rainbow neighbourhoods in $G$, cannot yield rainbow neighbourhoods in the graph $K_1+G$.

Finally, we note that the closed neighbourhood $N[K_1]$ consists of at least one coloured vertex of each colour in $(c_1,c_2,c_3,\ldots,c_{\ell},c_{\ell+1})$. Therefore, $r_\chi(K_1 + G) = 1 + r_\chi(G)$.
\end{proof}

Let $\mathcal{V}(G)$ be the collection of all subsets of the vertex set $V(G)$. Select any number, say $k$, of non-empty subsets $W_i\in \mathcal{V}(G)$ such that $\bigcup\limits_{i=1}^{k} W_i=V(G)$, with repetition of selection allowed. For the additional vertices $u_1,u_2,u_3, \ldots,u_k$, add the additional edges $u_iv_j$, $\forall v_j \in W_j$. The resultant graph is called a \textit{Chithra graph} of the given graph $G$, denoted by $\mathfrak{C}(G)$ (see \cite{KS1}). Hence, we have the following theorem.

\begin{theorem}\label{Lem-2.4}
If say, $t$ copies of $K_1$ are joined to $G$ to obtain a Chithra graph of $G$, then $r_\chi(t\cdot K_1 + G) = t + r_\chi(G)$.
\end{theorem}
\begin{proof}
The result is an immediate consequence of Lemma \ref{Lem-2.3}.
\end{proof}
	 
\begin{theorem}
Consider two graphs $G$ and $H$ of order $n_1,n_2$ respectively.
\begin{enumerate}\itemsep0mm
\item[(i)] $\begin{cases}
r_\chi(G\cup H)=r_\chi(G)+r_\chi(H); & \text{if}\quad \chi(G)=\chi(H);\\
r_\chi(G\cup H) < r_\chi(G) + r_\chi(H); & \text{otherwise}.
\end{cases}$
\item[(ii)] $r_\chi(G+H)=r_\chi(G)+r_\chi(H)$.

\item[(iii)] $\begin{cases}
r_\chi(G\circ H)=n_1(1+ r_\chi(H)); & \text{if}\quad \chi(H)\geq \chi(G)-1;\\
r_\chi(G\circ H) = r_\chi(G); & \text{otherwise}.
\end{cases}$
\end{enumerate}
\end{theorem}
\begin{proof} Let $\chi(G)=\ell_1$ and $\chi(H)=\ell_2$. Then, 

\vspace{0.2cm}
	
\ni \textit{Part (i):} If $\chi(G)=\chi(H)=\ell$, then both permit the same chromatic colouring on colours $(c_1,c_2,c_3,\dots,c_\ell)$. Therefore, in the disjoint union of $G$ and $H$, the operation cannot increase or decrease the respective values, $r_{\ell_1}$ and $r_{\ell_2}$. Hence, the result follows. Next, without loss of generality, let $\chi(G)<\chi(H)$, then no vertex $v\in V(G)$ can be adjacent to at least one of each colour found in $(c_1,c_2,c_3,\dots,c_{\ell_2})$. Therefore, $r_\chi(G\cup H) < r_\chi(G)+r_\chi(H)$.

\vspace{0.2cm}

\ni \textit{Part (ii):} Let the colouring $(c_1,c_2,c_3,\dots,c_{\ell_1})$ be a chromatic colouring of $G$ and $(c_{\ell_1+1},c_{\ell_1+2},c_{\ell_1+3},\ldots, c_{\ell_1+\ell_2})$ be a chromatic colouring of $H$. Then, we note that the colouring $(c_1,c_2,c_3,\ldots,c_{\ell_1},c_{\ell_1+1},c_{\ell_1+2},c_{\ell_1+3}, \ldots, c_{\ell_1+\ell_2})$ is a chromatic colouring of the graph $G+H$. Clearly, any vertex $v \in V(G)$ that yields a rainbow neighbourhood in $G$ also yields a rainbow neighbourhood in $G+H$ and vice versa. Also, any vertex $u \in V(G)$ that does not yield a rainbow neighbourhood in $G$, cannot yield a rainbow neighbourhood in $G+H$ and vice versa. Therefore, the result follows.

\vspace{0.2cm}

\ni \textit{Part (iii):} For any vertex $v \in V(G)$ with $c(v)=c_i$ recolour all vertices $u \in V(H)$, which have $c(u)=c_i$ to the colour $c_{\chi(H)+1}$. Then the first part of the result, $\chi(H) \geq \chi(G)-1$, is a direct consequence of Lemma \ref{Lem-2.3}.

\vspace{0.2cm}

Otherwise, it is clear that all vertices $v \in V(G)$ that yield a rainbow neighbourhood will yield a rainbow neighbourhood in $G\circ H$. Therefore, $r_\chi(G\circ H) \geq r_\chi(G)$. But, no vertex $w\in V(H)$ can yield a rainbow neighbourhood in $G\circ H$. This can be verified as follows. Assume that a vertex $w \in V(H)$ of the $t$-th copy of $H$ joined to $v \in V(G)$ is a vertex yielding a rainbow neighbourhood in $G\circ H$.  It means that vertex $w$ has at least one neighbour for each colour $c_i,\ 1 \leq i \leq \ell_2<\ell_1-1$ as well as the neighbour $v$ with the colour $c(v)=c_{\ell_2+1}$. Since, $c_{\ell_2+1}$ can at best be the colour $c_{\ell_1-1}$, the colour $c_{\ell_1} \notin N[w]$ in $r_\chi(G\circ H)$ which is a contradiction. Therefore, $r_\chi(G\circ H)=r_\chi(G)$. 
\end{proof}

\begin{corollary}
If $\chi(H)\geq \chi(G)-1$, then $\chi(G\circ H)=\chi(H)+1$ and $\chi(G\circ H) =\chi(G)$, otherwise.
\end{corollary}

Note that the number of rainbow neighbourhoods of a graph is independent from degree parameters or number of edges or number of vertices. Also, the size of the rainbow neighbourhoods may differ. 

\vspace{0.2cm}

To illustrate these observations, consider the cycle $C_5$ and label the vertices consecutively $v_1,v_2,v_3,v_4,v_5$. Any triple of consecutive vertices is a rainbow neighbourhood and each vertex hex has degree $2$. Consider the chromatic colouring $c(v_1)=c_1,c(v_2)=c_2,c(v_3)=c_1,c(v_4)=c_2,c(v_5)=c_3$. Attach any finite number $\ell \in \N$ of pendant vertices to say, $v_1$ to obtain the thorn cycle $C^\star_5$ and colour the pendants $c_2$. In the thorn cycle, the closed neighbourhood $N[v_1]$ is a rainbow neighbourhood and $d(v_1)=\ell+2$. The closed neighbourhood $N[v_4]$ is also a rainbow neighbourhood and all vertices have degree $2$. 

\vspace{0.2cm}

We observe that the Petersen graph, denoted $PG$, has $\chi(PG)=3$ and has the interesting property that only one vertex does not yield a rainbow neighbourhood. Hence, $r_\chi(PG) = 9$ whilst $|V(PG)|= 10$.

\section{On sum and product of rainbow neighbourhood number of a graph and its line graph}

We begin by presenting results on the sum and product of $r_\chi(G)$ and $r_\chi(L(G))$ for certain well known graph classes.

\begin{proposition}\label{Prop-2.6a}
For $n\ge 2$, $r_\chi(P_n) + r_\chi(L(P_n)) = 2n-1$ and $r_\chi(P_n)\cdot r_\chi(L(P_n)) = n(n-1)$.
\end{proposition}
\begin{proof}
Since $L(P_n)=P_{n-1}$, we have $r_\chi(P_n)=n$ and $r_\chi(L(P_n))=r_\chi(P_{n-1})=n-1$. Then, the result is immediate.
\end{proof}

\begin{proposition}\label{Prop-2.6b}
For $n\ge 3$, we have 
\begin{enumerate}\itemsep0mm 
\item[(i)] $r_\chi(C_n) + r_\chi(L(C_n))=
\begin{cases}
6, & \text {if $n$ is odd},\\
2n, & \text {if $n$ is even}.
\end{cases}$ 

\item[(ii)] $r_\chi(C_n)\cdot r_\chi(L(C_n)) =
\begin{cases}
9, & \text {if $n$ is odd},\\
n^2, & \text {if $n$ is even.}
\end{cases}$
\end{enumerate}
\end{proposition}
\begin{proof}
Since $L(C_n)=C_n$, we have  $r_\chi(L(C_n))=r_\chi(C_n)=n$ if $n$ is even and $r_\chi(L(C_n))=r_\chi(C_n)=3$ if $n$ is odd. Then, the result is straight forward.
\end{proof}

\begin{proposition}\label{Prop-2.6c}
Ladders, $L_n$, $n \geq 3$ that: $r_\chi(L_n) +r_\chi(L(L_n)) = 5n -4$ and $r_\chi(L_n)\cdot r_\chi(L(L_n))  = 2n(3n -4)$.
\end{proposition}
\begin{proof}
In view of Proposition \ref{Prop-2.1f}, we have that $r_\chi(L_n)= 2n$. By replacing each edge along the two poles, respectively and along the steps, the line graph is a $n-2$ layered graph as depicted in Figure \ref{fig:fig-1}.
	
\begin{figure}[h!]
\centering
\includegraphics[width=0.4\linewidth]{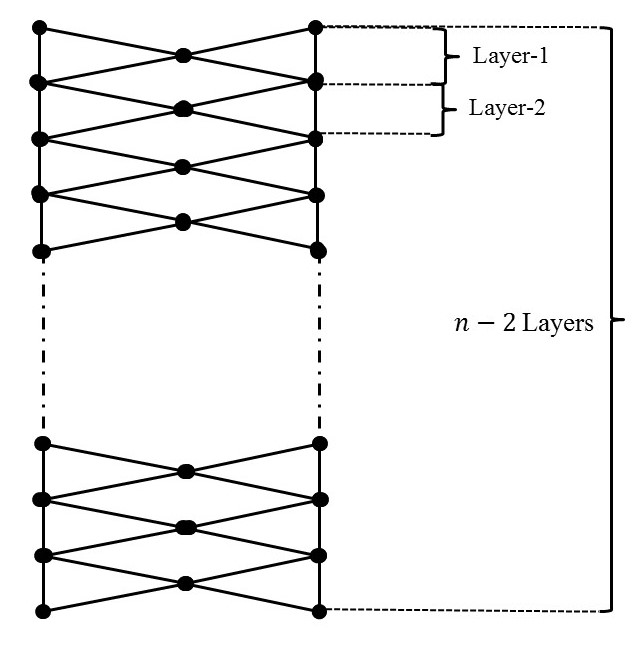}
\caption{}\label{fig:fig-1}
\end{figure}
	
It immediately follows that $r_\chi(L(L_n))=3n-4$. Therefore, $r_\chi(L_n) + r_\chi(L(L_n))=5n-4$ and $r_\chi(L_n)\cdot r_\chi(L(L_n))=2n(3n-4)$.
\end{proof}

\begin{proposition}\label{Prop-2.6d}
For $n\ge 3$, we have 
\begin{enumerate}
\item[(i)] $r_\chi(W_n) + r_\chi(L(W_n)) =
\begin{cases}
10, & \text {if $n = 3$},\\
n+4, & \text {if $n \geq 5$ and odd,}\\
2n+1, & \text {if $n \geq 4$ and even}.
\end{cases}$ 
\item[(ii)] $r_\chi(W_n)\cdot r_\chi(L(W_n)) =
\begin{cases}
24, &\text {if $n=3$},\\
4n, &\text {if $n \geq 5$ and odd,}\\
n(n+1), &\text {if $n \geq 4$ and even.}
\end{cases}$	
\end{enumerate}
\end{proposition}
\begin{proof}
The line graph of any wheel is structured with a complete graph $K_n$ as inner core. The inner core results from the spokes of the wheel since they are all incident with the central vertex. Each edge of the cycle of the wheel is adjacent to two other edges of the cycle as well as to spokes of the wheel, so the outer structure of the line graph is firstly, that of triangles to form a (complete) $n$-sun graph and secondly the \textit{tips} of the $n$-sun graph are linked to form a cycle $C_n$.

\vspace{0.2cm}

\ni The general structure of the line graph $L(W_n))$ is depicted in Figure \ref{fig:fig-2}.

\begin{figure}[h!]
\centering
\includegraphics[width=0.4\linewidth]{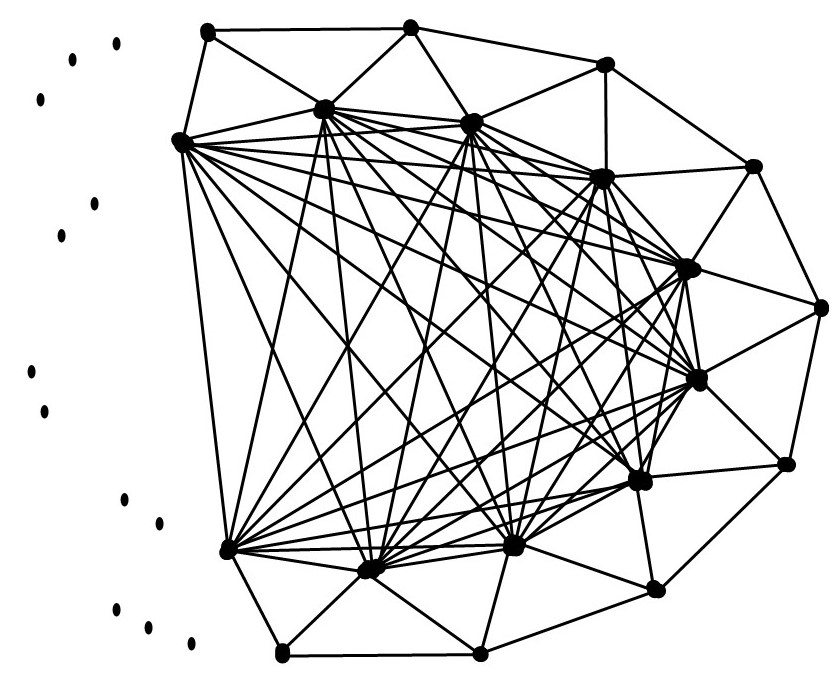}
\caption{}\label{fig:fig-2}
\end{figure}

\ni Hence, in conjunction with Proposition \ref{Prop-2.1e} all results follow immediately.
\end{proof}

Note that all the possible $\ell$-partite graph structures were not exhausted in Proposition \ref{Prop-2.6c}. It seems that further studies in this area are required. We introduce a new construction method to obtain a new graph associated to a given graph $G$, called the \textit{expanded line graph} of $G$, as follows.

\begin{enumerate}\itemsep0mm
\item[(a)] Label the edges of the graph $G$ as $e_1,e_2,e_3,\dots,e_{\varepsilon(G)}$.
\item[(b)]Replace each vertex $v\in V(G)$ with a complete graph $K_t$, $t=d_G(v)$ such that each distinct vertex of the complete graph is inserted into a distinct edge incident with vertex $v$. Hence, each edge $e_i \in E(G)$ will have two new vertices inserted. The complete graph $K_{d_G(v)}$ is called the \textit{$v$-clique} of vertex $v$.
\item[(c)]For each edge $e_i$, label the new inserted vertices $u_{i,1}$ and $u_{i,2}$.
\item[(d)]Connect the pairs of vertices $u_{i,1},u_{i,2}$ with a broken line edge.
\end{enumerate}

The expanded line graph of $G$ is denoted by $L^{\cdot \cdot}(G)$. Clearly by contracting all broken line edges hence, by merging all vertices $u_{i,1}$ and $u_{i,2}$ for $1\leq i \leq \varepsilon(G)$ we obtain the line graph $L(G)$.

\vspace{0.2cm}

Figure \ref{fig:fig-3} depict a graph $G$ and the corresponding expanded line graph $L^{\cdot \cdot}(G)$. Similar to the notion of a simple graph, it is important to note that two distinct cliques say, the $v$-clique and the $u$-clique can be linked by at most one broken line edge.

\begin{figure*}[h!]
\centering
\begin{subfigure}[b]{0.45\textwidth}
\centering
\includegraphics[height=2in]{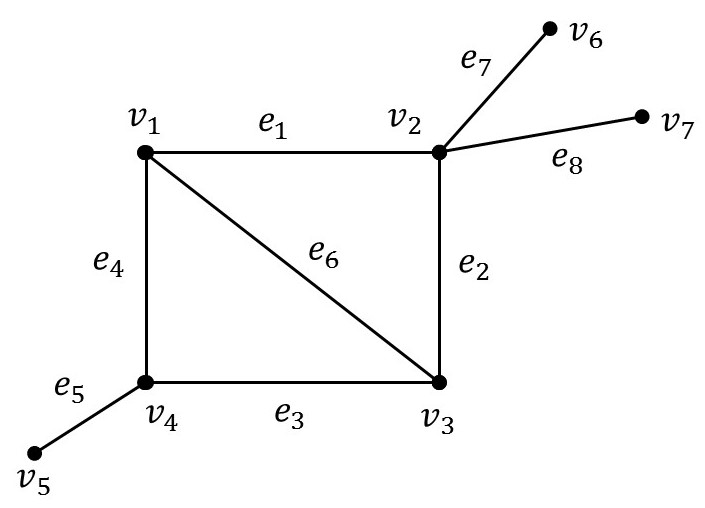}
\caption{$G$.}\label{fig:fig-1a}
\end{subfigure}%
\quad 
\begin{subfigure}[b]{0.45\textwidth}
\centering
\includegraphics[height=2in]{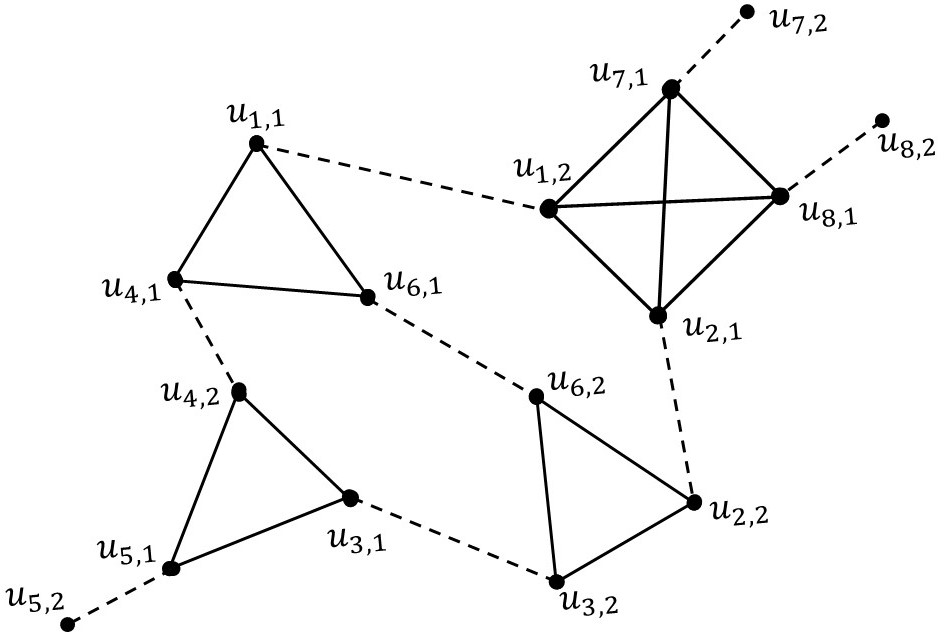}
\caption{$L^{\cdot \cdot}(G)$.}\label{Fig-1b}
\end{subfigure}
\caption{}\label{fig:fig-3}
\end{figure*}

\begin{theorem}\label{Thm-2.7}
For any graph $G$, $\Delta(G) \geq 3$, a vertex $w \in V(L(G))$, yields a rainbow neighbourhood (in $L(G)$)  if and only if $u_{i,1}$ or $u_{i,2}$ from which $w$ resulted, is a vertex of a maximum $v$-clique of some $v \in V(G)$.
\end{theorem}
\begin{proof}
Both the necessary and sufficient conditions follows directly from considering the expanded line graph.

If $G$ is 3-regular then all maximal cliques in $L(G)$ are $K_3$ (triangles) hence, the clique number is $\omega(L(G)) = 3$ and the result follows immediately because pairwise, distinct cliques share at most one common vertex. Now assume graph $G$ has $m$ vertices of maximum degree $\Delta(G) \neq \delta (G)$. Hence, in the line graph, $m$ maximum cliques $K_{\Delta(G)}$ exist, and only the vertices of these maximum cliques can yield a rainbow neighbourhood on the colours $(c_1,c_2,c_3,\dots,c_{\Delta(G)})$ in $L^{\cdot \cdot}(G)$. Clearly after contracting the broken line edges these vertices remain the same and $\chi(L(G)) = \Delta(G)$ remains because pairwise, two distinct cliques share at most one common vertex. Therefore, the result.
\end{proof}

\begin{corollary}\label{Cor-2.8}
Let a graph $G$ have $\ell$ vertices of degree $\Delta(G)$ then $r_\chi(L(G)) \leq \ell\cdot \Delta(G)$.
\end{corollary}

\begin{corollary}\label{Cor-2.9}
Consider any $t$-regular graph $G$, $t\geq 3$. Let the order of $G$ be $n$. Then for the line graph we have 
\begin{enumerate}\itemsep0mm
\item[(i)] $r_\chi(L(G)) = \varepsilon(G)$ which implies that:\\
\item[(ii)] $r_\chi(G) + r_\chi(L(G)) = r_\chi(G) + \varepsilon(G)$ and $r_\chi(G)\cdot r_\chi(L(G)) =  r_\chi(G)\cdot \varepsilon(G)$.
\end{enumerate}
\end{corollary}

\begin{proposition}\label{Prop-2.6e}
For $n\ge 2$, we have $r_\chi(K_n) + r_\chi(L(K_n))= \frac{n}{2}(n+1)$ and $r_\chi(K_n)\cdot r_\chi(L(K_n))= \frac{n^2}{2}(n-1)$.
\end{proposition}
\begin{proof}
The line graph of $K_n$ is $2(n-2)$-regular on $\frac{n}{2}(n-1)$ vertices and generally very difficult to present graphically. The proof follows from Theorem \ref{Thm-2.7} and Corollary \ref{Cor-2.8}.
\end{proof}

\begin{proposition}\label{Prop-2.6f}
For a complete $\ell$-partite graph $K_{r_1,r_2,\dots,r_\ell}$, where $\ell \geq 2,\ r_i \geq 1$ and $r_i \geq r_{i+1}>r_\ell, 1\leq i \leq \ell-2$ we have
\begin{enumerate}\itemsep0mm
\item[(i)] $r_\chi(K_{r_1,r_2,\dots,r_\ell}) + r_\chi(L(K_{r_1,r_2,\dots,r_\ell})) = \sum\limits_{i=1}^{\ell}r_i + \frac{r_\ell}{2}\sum\limits_{i=1}^{\ell -1}r_i(\sum\limits_{i=1}^{\ell -1}r_1 - 1)$.
\item[(ii)] $r_\chi(K_{r_1,r_2,\dots,r_\ell})\cdot r_\chi(L(K_{r_1,r_2,\dots,r_\ell})) = \sum\limits_{i=1}^{\ell}r_i\cdot \frac{r_\ell}{2}\sum\limits_{i=1}^{\ell -1}r_i(\sum\limits_{i=1}^{\ell -1}r_1 - 1)$.
\end{enumerate}
\end{proposition}
\begin{proof}
The result is a direct consequence of Theorem \ref{Thm-2.7} and Corollary \ref{Cor-2.8}. Note that the complete $\ell$-partite graph under consideration has exactly $r_\ell$ maximum complete graphs in the corresponding line graph.
\end{proof}

Results from applying Theorem \ref{Thm-2.2} and Corollary \ref{Cor-2.9} to some important $2$-chromatic graphs are given in Table \ref{Tab-1}.

\begin{tabularx}{\textwidth}{|X|X|X|X|X|X|}
\hline
Graph $G$ & $\nu(G)$ & $\varepsilon(G)$ & Degree regularity & $r_\chi(G)+r_\chi(L(G))$ & $r_\chi(G)\cdot r_\chi(L(G))$\\
\hline
Iofinova-Ivanov & 110 & 165 & 3 & 275 & 18150\\
\hline
Balaban 10-cage & 70 & 105 & 3 & 175 & 7350\\
\hline
Cubicle & 8 & 12 & 3 & 20 & 96\\
\hline
Dyck & 32 & 48 & 3 & 80 & 1536\\
\hline
Ellingham-Horton  & 54 (78) & 81 (167) & 3 & 135(245) & 4374(13026)\\
\hline
$F_26A$ & 26 & 39 & 3 & 65 & 1014\\
\hline
Folkman & 20 & 40 & 4 & 60 & 800\\
\hline
Foster & 90 & 135 & 3 & 225 & 12150\\
\hline
Franklin & 12 & 18 & 3 & 30 & 216\\
\hline
Gray & 54 & 81 & 3 & 135 & 4374\\
\hline
Harries & 70 & 105 & 3 & 175 & 7350\\
\hline
Heawood & 14 & 21 & 3 & 35 & 294\\
\hline
Hoffman & 16 & 32 & 4 & 48 & 512\\
\hline
Horton & 96 & 144 & 3 & 240 & 13824\\
\hline
Ljubljana & 112 & 168 & 3 & 280 & 18816\\
\hline
Naura & 24 & 36 & 3 & 60 & 864\\
\hline
Pappus & 18 & 27 & 3 & 45 & 486\\
\hline
Tutte-Coxeter & 30 & 45 & 3 & 75 & 1350\\
\hline
\caption{}\label{Tab-1}
\end{tabularx}

\subsection{Important observations}

Following from Theorem \ref{Thm-2.1} we have that $r_\chi(G) \geq \chi(G) \Rightarrow r_\chi(L(G)) \geq \chi(L(G))$ and $r_\chi(\overline{G}) \geq \chi(\overline{G})$. Hence, all known Nordhauss-Gaddum lower bounds apply in respect of the sum and the product of rainbow neighbourhood numbers for $G$, $L(G)$ and $\overline{G}$. So do other lower bounds for $\chi(G)$ and correspondingly for $L(G)$ and $\overline{G}$, apply. Note we adopt the convention that $\overline{K}_1 = K_1$ and $L(K_1) = K_1$. For a graph of order $n\geq 1$ and size $q\geq 0$, we list a few of these bounds below.

\begin{enumerate}\itemsep0mm 
\item[(i)] $2\sqrt{n} \leq r_\chi(G) + r_\chi(\overline{G}) \leq 2n$, and $n \leq r_\chi(G)\cdot r_\chi(\overline{G}) \leq n^2$.
\item[(ii)]  $2\leq r_\chi(G) + r_\chi(L(G)) \leq n +\ell \cdot \Delta(G)$, and $1 \leq r_\chi(G)\cdot r_\chi(L(G)) \leq n\ell\cdot \Delta(G)$.
\item[(iii)]  It is known that if graph $G$ is $t$-regular then $\chi(G) \geq \frac{n}{n-t}$ and therefore, $\chi(\overline{G}) \geq \frac{n}{t+1}$. Hence, $r_\chi(G) + r_\chi(\overline{G}) \geq \frac{n(n+1)}{(n-t)(t+1)}$ and $r_\chi(G)\cdot r_\chi(\overline{G}) \geq \frac{n^2}{(n-t)(t+1)}$.
\item[(iv)]  Let $\gamma(G)$ be the domination number of $G$. Since $\gamma(G) \leq n - \Delta(G)$ it follows that $r_\chi(G) + \gamma(G) \leq 2n- \Delta(G)$ and $r_\chi(G)\cdot \gamma(G) \leq n(n-\Delta(G))$. Also since, we have the upper bound $\gamma(G) \leq \lceil \frac{n+1-\delta(G)}{2}\rceil$ similar inequalities are found in terms of $\delta(G)$.
\end{enumerate} 

\begin{theorem}\label{Thm-2.10}
For a connected graph $G$ of order $n \geq 3$ we have $r_\chi(G) = \chi(G)$ if and only if $G$ is an odd cycle or complete.
\end{theorem}
\begin{proof}
It is clear that if $G$ is a cycle $C_n$, $n$ is odd or a complete graph, $K_n$, then $r_\chi(C_n) = 3 = \chi(C_n)$ and $r_\chi(K_n) = n = \chi(K_n)$.

Conversely, let $n=3$ then $G$ is either the path $P_3$, \'{o}r, the cycle $C_3$ or put alternatively, the complete graph $K_3$. Since $r_\chi(P_3) = 3 \neq 2 = \chi(P_3)$ the result, $G = C_3$ (or $K_3$)$\Rightarrow r_\chi(G) = \chi(G)$ holds. Similarly for a graph on $n = 4$ vertices the result cannot hold for $P_4$, $C_4$, star $S_{1,3}$ or the 1-chord cycle. It only holds for that $G = K_4 \Rightarrow r_\chi(K_4) = \chi(K_4) = 4$. Assume the result holds for all graphs $G$ of order $3 \leq \ell \leq t$ if $G$ is either $C_\ell$, $\ell$ is odd, or $K_\ell$, $\forall \ell \in \N$.

Consider any graph $G$ of order $t+1$. Clearly it holds that $G = C_{t+1}$, $t+1$ is odd $\Rightarrow r_\chi(G) = \chi(G) = 3$. Also, $G = K_{t+1} \Rightarrow r_\chi(G) = \chi(G) = t+1$. So let $G$ be any other graph of order $t+1$ and assume $r_\chi(G) = \chi(G)$. Now we have $\chi(G) \leq r_\chi(G) < t+1$. Hence, it is possible to find and remove one vertex which does not yield a rainbow neighbourhood and the chromatic number remains the same. Therefore, we obtain a new graph $G'$ of order $t$ such that $G' \Rightarrow r_\chi(G') = \chi(G')$. The aforesaid is a contradiction hence, no such $G$ exists. It implies that for any $G$ of order $t+1$, only $G = C_{t+1}$, $t+1$ is odd $\Rightarrow r_\chi(G) = \chi(G)$ or $G = K_{t+1}\Rightarrow r_\chi(G) = \chi(G)$.
\end{proof}

\begin{corollary}
Any connected graph of order $n$ that have $r_\chi(G) = \chi(G)$ is either 3-chromatic or $n$-chromatic.
\end{corollary}
\begin{proof}
The result follows directly from Theorem \ref{Thm-2.10}.
\end{proof}

Consider a connected graph $G$ on vertices $v_1,v_2,v_3,\dots,v_n$ and add $t$ isolated vertices. Now join each isolated vertex to vertices $v_2,v_3,v_4,\dots,v_n$ to obtain $G^*_t$. Since, $r_\chi(G) \geq \chi(G)$ an easy to state the next theorem.

\begin{theorem}\label{Thm-2.12}
For $a,b \in \N,\ a \geq b \geq 2$ there exists a graph $G$ such that $r_\chi(G) = a$, and $\chi(G) = b$.
\end{theorem}
\begin{proof}
Let $G = K_b$ and construct $G^*_{a-b}$. The result follows immediately.
\end{proof}

The graph constructed in Theorem \ref{Thm-2.12} is minimum in order and size as well. Note that $K_{b-1}+(a-b+1)K_1 = K^*_{b,(a-b)}$.

\section{Conclusion}

Inherent to the proofs found in Proposition \ref{Prop-2.1c} and Proposition \ref{Prop-2.1d},  where $n$ is even; and from Proposition \ref{Prop-2.1e}, and in Proposition \ref{Prop-2.1f},  we observe that if it is possible to permit a chromatic colouring of any graph $G$ of order $n$ such that the star subgraph obtained from vertex $v$ as center and its open neighbourhood $N(v)$ the pendant vertices, has at least one coloured vertex from each colour for all $v \in V(G)$ then $r_\chi(G) = n$. It remains an open problem to characterise such graphs, if possible. 

\vspace{0.2cm}

Determining the rainbow neighbourhood number under other known graph operations remains open for study. Also, finding an efficient algorithm to examine the existence of the property in a given graph will be a valuable contribution. 

\vspace{0.2cm}

Note that the Rainbow Neighbourhood Convention for colouring ensures that the minimum number of rainbow neighbourhoods are yielded. For example, if the cycle $C_9$ with vertices labeled consecutively, $v_1,v_2,v_3,\dots,v_9$ are coloured $c(v_1)=c_1,c(v_2)=2,c(v_3)=c_3,c(v_4)=1,c(v_5)=c_2,c(v_6)=c_3,c(v_7)=c_1,c(v_8)=c_2,c(v_9)=c_3$, the number of rainbow neighbourhoods are $9$. We denote this as $r^+_\chi(C_9) = 9$. Studying the parameter $r^+_\chi(G)$ and the relationship thereof with $r_\chi(G)$ is a new field for further investigation.

\vspace{0.2cm}

It will be interesting to study possible results in respect of $r_\chi(G)$, similar to certain results found in \cite{AES,RLB} and for the graphical embodiments discussed in \cite{KSC}.

\end{document}